\documentclass[psamsfonts]{amsart}

\usepackage{amssymb,amsfonts,amscd}
\usepackage[all,arc]{xy}
\usepackage{enumerate}
\usepackage{mathrsfs}

\newtheorem{thm}{Theorem}[section]
\newtheorem{cor}[thm]{Corollary}
\newtheorem{prop}[thm]{Proposition}
\newtheorem{lem}[thm]{Lemma}
\newtheorem{conj}[thm]{Conjecture}

\newtheorem*{thm1}{Theorem A}

\theoremstyle{definition}
\newtheorem{defn}[thm]{Definition}

\newtheorem{exmp}[thm]{Example}

\theoremstyle{remark}
\newtheorem{rem}[thm]{Remark}

\makeatletter
\let\c@equation\c@thm
\makeatother
\numberwithin{equation}{section}

\title[]{On the escape rate of geodesic loops in an open manifold with nonnegative Ricci curvature}

\author[]{Jiayin Pan}
\thanks{}

\newcommand{\Addresses}{{ 
		\bigskip
		\footnotesize
		
		Jiayin Pan, \textsc{Department of Mathematics, University of California\ -\ Santa Barbara, CA, USA.}\par\nopagebreak
		\textit{E-mail address}: \texttt{jypan10@gmail.com}

}}

\begin{document}
	
	\begin{abstract}
	A consequence of the Cheeger-Gromoll splitting theorem states that for any open manifold $(M,x)$ of nonnegative Ricci curvature, if all the minimal geodesic loops at $x$ that represent elements of $\pi_1(M,x)$ are contained in a bounded ball, then $\pi_1(M,x)$ is virtually abelian. We generalize the above result: if these minimal representing geodesic loops of $\pi_1(M,x)$ escape from any bounded metric balls at a sublinear rate with respect to their lengths, then $\pi_1(M,x)$ is virtually abelian. 
	\end{abstract}
	
\maketitle

One of the most basic theorems in Riemannian manifolds with nonnegative Ricci curvature is the Cheeger-Gromoll splitting theorem \cite{CG_split}. The splitting theorem has an important application on the structure of fundamental groups: for any closed manifold of $\mathrm{Ric}\ge 0$, its fundamental group must be virtually abelian, that is, $\pi_1(M)$ contains an abelian subgroup of finite index \cite{CG_split}. This beautiful result has an analog for open manifolds (non-compact and complete), which we describe below. Let $M$ be an open manifold of $\mathrm{Ric}\ge 0$ and let $x\in M$. For each element $\gamma\in\pi_1(M,x)$, among all the loops based at $x$ that represent $\gamma$, we can find one with the minimal length. This loop, denoted as $c_\gamma$, is a geodesic loop based at $x$. By applying the Cheeger-Gromoll splitting theorem, it can be shown that if there are $x\in M$ and $R>0$ such that $c_\gamma$ is contained in $B_R(x)$, the metric ball of radius $R$ and center $x$, for all $\gamma\in \pi_1(M,x)$, then $\pi_1(M)$ is virtually abelian (we include a proof of this statement in Appendix \ref{app_bdd} for readers' convenience). Note that when $M$ has nonnegative sectional curvature, then by the soul theorem and the Sharafutdinov retraction \cite{CG_soul,Sha}, it is not difficult to see that this condition always holds by setting $x$ in a soul. However, this result is quite limited: it cannot be applied to positive Ricci curvature. In fact, if $M$ has $\mathrm{Ric}>0$ and an infinite fundamental group, then $c_\gamma$ will escape from any bounded balls as $\gamma$ exhausts $\pi_1(M,x)$ \cite{SW}.

We introduce a quantity that measures how fast these minimal representing geodesic loops of $\Gamma=\pi_1(M,x)$ escape from bounded balls centered at $x$. We define the escape rate of $(M,x)$ by comparing the radius of the smallest closed ball centered at $x$ and containing $c_\gamma$, to the length of $c_\gamma$. Recall that since $\Gamma$ acts isometrically on the Riemannian universal cover $(\widetilde{M},\tilde{x})$ as covering transformations, for any $\gamma\in \Gamma$, the length of $c_\gamma$ equals $d(\tilde{x},\gamma\tilde{x})$, written as $|\gamma|$ when the base point $x$ is clear. Note that if $\Gamma$ is infinite, then $\sup_{\gamma\in \Gamma} |\gamma|=\infty$.

\begin{defn}
	Let $(M,x)$ be an open manifold with an infinite fundamental group. We define the \textit{escape rate} of $(M,x)$, a scaling invariant, as
	$$E(M,x)=\limsup_{|\gamma|\to \infty} \dfrac{d_H(x,c_\gamma)}{|\gamma|},$$
	where $\gamma\in\pi_1(M,x)$ and $d_H$ is the Hausdorff distance.
\end{defn} 

From the definition, it is clear that $E(M,x)\le 1/2$. For an open manifold as a warped product $[0,\infty)\times_f S^{p-1}\times_h N$, where $N$ is a closed manifold with an infinite fundamental group, we can estimate its escape rate and see that $E(M,x)$ is determined by the decay rate of the warping function $h$ (see Appendix \ref{app_exmp}).

When $E(M,x)=0$, the minimal representing geodesic loops of $\pi_1(M,x)$ either escape from bounded balls at a sub-linear rate with respect to their lengths, or they are all contained in a bounded ball.
It can be shown that whether $E(M,x)$ is zero does not depend on the choice of $x\in M$ (see Corollary \ref{cor_zero}). 

We state the main theorem of this paper.

\begin{thm1}\label{main}
	Let $(M,x)$ be an open $n$-manifold of $\mathrm{Ric}\ge 0$. If $E(M,x)=0$, then $\pi_1(M,x)$ is virtually abelian.
\end{thm1}	

The converse of Theorem A is not true (see Appendix \ref{app_exmp}). 

In general, an open manifold of $\mathrm{Ric}\ge 0$ may not have a virtually abelian fundamental group. Wei showed that any finitely generated torsion-free nilpotent group can be realized as the fundamental group of an open manifold of positive Ricci curvature \cite{Wei}. Based on Wei's construction, Wilking generalized this result to any finitely generated nilpotent group \cite{Wilk}.

Regarding the general structure of fundamental groups of open manifolds with $\mathrm{Ric}\ge 0$, Milnor showed that any finitely generated subgroup of $\pi_1(M,x)$ has polynomial growth \cite{Mil}. Combined with Gromov's work \cite{Gro_poly}, such a subgroup must be virtually nilpotent. See \cite{KW} for the index bound on the nilpotent subgroup. We mention that it is not difficult to show that $E(M,x)<1/2$ implies the finite generation of $\pi_1(M,x)$ (see Lemma \ref{fg}).

The contrapositive of Theorem A gives a geometric characterization of open manifolds with $\mathrm{Ric}\ge 0$ and non-virtually-abelian fundamental groups: if $\pi_1(M,x)$ contains a torsion-free nilpotent non-abelian subgroup, then $E(M,x)>0$, that is, there is a sequence $\gamma_i\in \pi_1(M,x)$ such that $|\gamma_i|\to\infty$ and $d_H(x,c_{\gamma_i})$ is proportional to the length of $c_{\gamma_i}$ (also see Conjecture \ref{quest_gap}). This also explains why we cannot use certain warping functions to construct an open manifold as a warped product with $\mathrm{Ric}\ge 0$ and a torsion-free nilpotent non-abelian fundamental group (see Appendix \ref{app_exmp}).

In \cite[Corollary 4.2]{Pan_al_stable}, it was shown that if $\widetilde{M}$ has certain geometric stability condition at infinity, (for example, $\widetilde{M}$ has a unique asymptotic cone as a metric cone) then $\pi_1(M)$ is virtually abelian. Based on \cite{Pan_eu,Pan_al_stable} and the results in this paper, we can show that these manifolds satisfy the condition $E(M,x)=0$ (see Corollary \ref{cor_cone_zero}).

We briefly explain the strategy to prove Theorem A. The principle is, when viewing $(\widetilde{M},\tilde{x},\Gamma)$ from afar, we cannot distinguish the bounded case ($\sup_{\gamma\in\Gamma} d_H(x,c_\gamma)$ is finite) and the sublinear case ($E(M,x)=0$). To illustrate this approach rigorously, we study the asymptotic geometry of $(\widetilde{M},\tilde{x},\Gamma)$ via equivariant Gromov-Hausdorff convergence \cite{FY}:
$$(r_i^{-1}\widetilde{M},\tilde{x},\Gamma)\overset{GH}\longrightarrow (Y,y,G),$$
where $r_i\to\infty$ and $G$ acts isometrically and effectively on $Y$. The limit space $(Y,y,G)$ above is called an \textit{equivariant asymptotic cone} of $(\widetilde{M},\Gamma)$. We prove the following theorem on the relations between $E(M,x)$ and any equivariant asymptotic cone of $(\widetilde{M},\Gamma)$.

\begin{thm}\label{thm_equivalent}
	Let $(M,x)$ be an open manifold of $\mathrm{Ric}\ge 0$. Then the following statements are equivalent:\\
	(1) $E(M,x)=0$;\\
	(2) the orbit $\Gamma\cdot\tilde{x}$ is weakly asymptotically geodesic;\\
	(3) for any equivariant asymptotic cone $(Y,y,G)$ of $(\widetilde{M},\Gamma)$, the orbit $G\cdot y$ is geodesic in $Y$;\\
	(4) for any equivariant asymptotic cone $(Y,y,G)$ of $(\widetilde{M},\Gamma)$, the orbit $G\cdot y$ is geodesic in $Y$ and is isometric to a standard Euclidean space.
\end{thm}

The notion \textit{weakly asymptotically geodesic} in (2) will be introduced later in Section \ref{sec_geod}. This extends and is inspired by the notion \textit{asymptotically geodesic} (see Section \ref{sec_pre}.2).

The technical part of Theorem \ref{thm_equivalent} is to prove (3)$\Rightarrow$(4). Two ingredients are crucial in this part. The first one is the Cheeger-Colding splitting theorem \cite{CC1}, which is a substantial generalization of the Cheeger-Gromoll splitting theorem to the case of Ricci limit spaces. With this splitting theorem, we can reduce the orbit $G\cdot y$ to a metric product $\mathbb{R}^k\times Z$, where $Z$ is compact (see Lemma \ref{orbit_eu_cpt}). The second one is a critical rescaling argument, which was first developed in \cite{Pan_eu} by the author, to rule out the compact factor $Z$ in $G\cdot y$. The geometric intuition behind the argument here is quite different from the one in \cite{Pan_eu} (see Remark \ref{rem_crit_scal} for explanations).

To deduce Theorem A from Theorem \ref{thm_equivalent}, we further consider a nilpotent subgroup $N$ of $\Gamma$ of finite index. Note that the existence of such a subgroup is guaranteed by \cite{Mil,Gro_poly}. We consider the convergence
$$(r_i^{-1}\widetilde{M},\tilde{x},\Gamma,N)\overset{GH}\longrightarrow (Y,y,G,H).$$
Combined with Theorem \ref{thm_equivalent}, we show that the limit group $H$ acts as translations in the Euclidean factor of $Y$ (see Corollary \ref{nil_trans}). Back to the geometry of $(\widetilde{M},\tilde{x},\Gamma)$, this implies that any element $\gamma \in N$ with large displacement at $\tilde{x}$ acts almost as a translation at $\tilde{x}$, that is,
$$d(\gamma^2\tilde{x},\tilde{x})\ge 1.9\cdot  d(\gamma\tilde{x},\tilde{x}).$$
Applying an argument in \cite[Section 4]{Pan_al_stable}, the virtually abelian structure follows.

We organize the paper as follows. In Section \ref{sec_pre}, we go through some preliminaries and background. We prove Theorem \ref{thm_equivalent} (1)$\Rightarrow$(2)$\Rightarrow$(3) in Section \ref{sec_geod}. Then we finish the proof of Theorem \ref{thm_equivalent} and Theorem A in Section \ref{sec_eu}. Appendix \ref{app_bdd} includes a proof of the bounded case by the Cheeger-Gromoll splitting theorem. In Appendix \ref{app_exmp}, we estimate the escape rates of some known and new examples of open manifolds with positive Ricci curvature.

\tableofcontents

Acknowledgment: The author is partially supported by AMS Simons travel fund. The author would like to thank Guofang Wei for many helpful discussions.

\section{Preliminaries}\label{sec_pre}

\noindent\textbf{1.1 Gromov-Hausdorff convergence and asymptotic cones}

Let $(M,x)$ be an open manifold of $\mathrm{Ric}\ge 0$. For any sequence $r_i\to\infty$, passing to a subsequence if necessary, we obtain the following pointed Gromov-Hausdorff convergence \cite{Gro_book}:
$$(r_i^{-1}M,x)\overset{GH}\longrightarrow (Z,z).$$
We call the above $(Z,z)$ an \textit{asymptotic cone} of $M$, or a tangent cone of $M$ at infinity. The limit space $(Z,z)$ in general is not unique; in other words, it may depend on the scaling sequence $r_i$ \cite{CC2}. $(Z,z)$ does not depend on the base point $x$.

Cheeger and Colding \cite{CC1} proved a splitting theorem for Ricci limit spaces that comes from a sequence of complete manifolds with $\mathrm{Ric}\ge -\epsilon_i\to 0$, which substantially generalizes the Cheeger-Gromoll splitting theorem. In the context of asymptotic cones of open manifolds with $\mathrm{Ric}\ge 0$, we obtain the following result.

\begin{thm}\cite{CC1}
	Let $(M,x)$ be an open manifold of nonnegative Ricci curvature and let $(Z,z)$ be an asymptotic cone of $(M,x)$. Suppose that $Z$ contains a line, then $Z$ splits isometrically as $\mathbb{R}\times Z'$, where $Z'$ is a length metric space.
\end{thm}

To study the asymptotic geometry of $\pi_1(M,x)$-action on the Riemannian universal cover $(\widetilde{M},\tilde{x})$, we use the equivariant Gromov-Hausdorff convergence, which was introduced by Fukaya and Yamaguchi \cite{FY}. Letting $r_i\to\infty$, we can always obtain a convergent subsequence as below.
\begin{center}
	$\begin{CD}
	(r^{-1}_i\widetilde{M},\tilde{x},\Gamma) @>GH>> 
	(Y,y,G)\\
	@VV\pi V @VV\pi V\\
	(r^{-1}_iM,x) @>GH>> (Z,z).
	\end{CD}$
\end{center}
We call $(Y,y,G)$ an \textit{equivariant asymptotic cone} of $(\widetilde{M},\Gamma)$. The limit group $G$ is a Lie group \cite{CC3,CN} that acts isometrically and effectively on $Y$. $Z$ is isometric to the quotient space $Y/G$ \cite{FY}.\\

\noindent\textbf{1.2 Asymptotically geodesic metrics}

Part of the work in this paper is motivated by the following concept in geometric group theory (see \cite[(34)]{Pansu}):

\begin{defn}\label{def_asy_geo}
	Let $\Gamma$ be a finitely generated group with a left-invariant metric $\rho$. We say that $(\Gamma,\rho)$ is \textit{asymptotically geodesic}, if for any $\epsilon>0$, there is $s=s(\epsilon)>0$ such that for any $\gamma\in \Gamma$, we can find a word $\prod_{j=1}^N \gamma_j=\gamma$ satisfying
	$$\sum_{j=1}^N\rho(e,\gamma_j)\le (1+\epsilon)\rho(e,\gamma)$$
	and $\rho(e,\gamma_j)\le s$ for all $j$.
\end{defn}

For instance, any word metric on $\Gamma$ is asymptotically geodesic by choosing $s(\epsilon)=1$ for all $\epsilon>0$.

For a Riemannian manifold $(M,x)$, since $\Gamma=\pi_1(M,x)$ acts on the Riemannian universal cover $\widetilde{M}$ freely and isometrically, we can define a natural left-invariant metric on $\Gamma$ from its orbit at $\tilde{x}$, that is,
$$\rho(\gamma_1,\gamma_2)=d(\gamma_1\tilde{x},\gamma_2\tilde{x}).$$
It is well-known that when $M$ is closed, $\rho$ is asymptotically geodesic and is bi-Lipschitz equivalent to any word metric on $\Gamma$ (see \cite[(34)]{Pansu}).

When $\Gamma$ has polynomial growth, Gromov proved that for any sequence $r_i\to\infty$, the rescaled sequence $(r_i^{-1}\Gamma,\rho,e)$ is precompact in the pointed Gromov-Hausdorff topology \cite{Gro_poly}. Consequently, $(\Gamma,\rho)$ has asymptotic cones as length metric spaces. See \cite{Pansu} for the detailed descriptions of these asymptotic cones.

For an open manifold $(M,x)$, if $\sup_{\gamma\in \Gamma} d_H(x,c_\gamma)$ is finite, then it is not difficult to see that the corresponding $(\Gamma,\rho)$ is asymptotically geodesic as well. We will define a notion called \textit{weakly asymptotic geodesic} below that extends Definition \ref{def_asy_geo} and corresponds to the sublinear case $E(M,x)=0$ (see Definition \ref{def_weak_asy_geo}).

\section{Geodesic orbit in the asymptotic cone}\label{sec_geod}

In this section, we introduce the notion \textit{weakly asymptotically geodesic} and show that (1)$\Rightarrow$(2)$\Rightarrow$(3) in Theorem \ref{thm_equivalent}.

Before doing so, we prove a result that mentioned in the introduction.

\begin{lem}\label{fg}
	If $\pi_1(M,x)$ is not finitely generated, then $E(M,x)=1/2$.
\end{lem}

\begin{proof}
	If $\pi_1(M,x)$ is not finitely generated, then one can choose a sequence of generators $\{\gamma_1,...\gamma_i,..\}$ such that $c_i$, the minimal representing geodesic loop of $\gamma_i$ is minimal up to halfway \cite[Lemma 5]{Sor}. In other words, each $\gamma_i$ satisfies
	$$\dfrac{d_{H}(x,c_i)}{|\gamma_i|}=\dfrac{1}{2}.$$
	The result follows immediately. 
\end{proof}

It is conjectured by Milnor that $\pi_1(M,x)$ is always finitely generated when $M$ has nonnegative Ricci curvature \cite{Mil}.

\begin{defn}\label{def_weak_asy_geo}
	Let $\Gamma$ be a finitely generated group with a left-invariant metric $\rho$. We say that $(\Gamma,\rho)$ is \textit{weakly asymptotically geodesic}, if there is a function $s(\epsilon,R)$ with the properties below:\\
	(1) for every fixed $\epsilon$, $R/s(\epsilon,R)\to \infty$ as $R\to\infty$;\\
	(2) for any $\epsilon>0$, $R>0$ and any $\gamma\in \Gamma$ with $\rho(e,\gamma)=R$, we can find a word $\prod_{j=1}^N \gamma_j=\gamma$ such that
	$$\sum_{j=1}^N\rho(e,\gamma_j)\le (1+\epsilon)\rho(e,\gamma)$$
	and $\rho(e,\gamma_j)\le s(\epsilon,R)$ for all $j$.
\end{defn}

%

\begin{prop}\label{asy_geo}
	Let $(M,x)$ be an open manifold and let $(\widetilde{M},\tilde{x})$ be its Riemannian universal cover. If $E(M,x)=0$, then $(\Gamma,\rho)$ is weakly asymptotic geodesic, where $\rho$ is given by the orbit $\Gamma\cdot \tilde{x}$ as $\rho(\gamma_1,\gamma_2)=d(\gamma_1\tilde{x},\gamma_2\tilde{x})$ for all $\gamma_1,\gamma_2\in\Gamma$.
\end{prop}

\begin{proof}
   Let $\Gamma(R)=\{\gamma \in \Gamma|d(\tilde{x},\gamma\tilde{x})\le R \}$ and let
   $$D(R)=\max_{\gamma\in \Gamma(R)} d_H(x,c_\gamma).$$
   It is clear that $E(M,x)=0$ implies that $D(R)/R\to 0$ as $R\to\infty$.
   Let $\epsilon>0$. We define 
   $$s(\epsilon,R)=2(\epsilon^{-1}+1)D(R).$$
   For any fixed $\epsilon>0$, $R/s(\epsilon,R)\to \infty$ as $R\to\infty$.
   
   Let $\gamma\in \Gamma$ with $R=\rho(e,\gamma)$. If $R\le s(\epsilon,R)$, condition (2) in Definition \ref{def_weak_asy_geo} holds trivially. We assume that $R>s(\epsilon,R)$ below. Let $c$ be a minimal geodesic loop based at $x$ that represents $\gamma$. $c$ is contained in $\overline{B_{D(R)}}(x)$. Lifting $c$ to the universal cover $(\widetilde{M},\tilde{x})$, we obtain a minimal geodesic $\tilde{c}$ from $\tilde{x}$ to $\gamma\tilde{x}$ contained in $\pi^{-1}(\overline{B_{D(R)}}({x}))$ and of length $R$, where $\pi:\widetilde{M}\to M$ is the covering map. Re-parameterizing $\tilde{c}$ if necessary, we assume that $\tilde{c}: [0,R]\to \widetilde{M}$ is of unit speed. We choose a series of points $\{\tilde{c}(t_j)\}_{j=0}^N$ on $\tilde{c}$ such that $t_0=0$, $t_N=R$ and
   $$t_{j}-t_{j-1}=2\epsilon^{-1}D(R) \text{ for } j=1,...,N-1, \quad t_N-t_{N-1}\le 2\epsilon^{-1}D(R).$$
   It is clear that 
   $$N\le 1+\frac{R}{2\epsilon^{-1}D(R)}.$$ 
   Because $\pi^{-1}(\overline{B_{D(R)}}(x))$ is $\Gamma$-invariant with a compact quotient $\overline{B_{D(R)}}(x)$, for each $j$ there is $\beta_j\in \Gamma$ with $d(\beta_j\tilde{x},\tilde{c}(t_j))\le D(R)$; we would always choose $\beta_0=e$ and $\beta_N=\gamma$. Let $\gamma_j=\beta_{j-1}^{-1} \beta_j $ for $j=1,...,N$. Then $\gamma=\prod_{j=1}^N \gamma_j$. Note that for $j=1,...,N$,
   \begin{align*}
   \rho(e,\gamma_j)&=d(\beta_{j-1}\tilde{x},\beta_j\tilde{x})\\
   &\le d(\beta_{j-1}\tilde{x},\tilde{c}(t_{j-1}))+d(\tilde{c}(t_{j-1}),\tilde{c}(t_j))+d(\tilde{c}(t_j),\beta_j\tilde{x})\\
   &\le D(R)+2\epsilon^{-1}D(R)+D(R)\\
   &=s(\epsilon,R).
   \end{align*}
   Moreover, 
   \begin{align*}
   &\sum_{j=1}^N \rho(e,\gamma_j)=\sum_{j=1}^N d(\beta_{j-1}\tilde{x},\beta_j\tilde{x})\\
   \le& (D(R)+|t_1-t_0|)+\sum_{j=1}^{N-2} (2D(R)+|t_{j+1}-t_j|)+ (D(R)+|t_N-t_{N-1}|)\\
   =&2(N-1)D(R)+R\\
   \le&2\frac{R}{2\epsilon^{-1}D(R)}D(R)+R\\
   =&(1+\epsilon)\rho(e,\gamma).
   \end{align*}
   This completes the proof.
\end{proof}

\begin{defn}
	Let $(X,d)$ be a length metric space and let $Z$ be a closed subset of $X$. We say that $Z$ is \textit{geodesic} in $X$ if the intrinsic metric on $Z$ coincides with the extrinsic one; in other words, for any $z_1,z_2\in Z$, there is a minimal geodesic connecting them that is contained in $Z$.
\end{defn}

\begin{prop}\label{orbit_geo}
	Let $(M,x)$ be an open $n$-manifold of $\mathrm{Ric}\ge0$. If the orbit $\Gamma\cdot\tilde{x}$ on the universal cover is weakly asymptotically geodesic, then for any equivariant asymptotic cone $(Y,y,G)$ of $(\widetilde{M},\Gamma)$, the limit orbit $G\cdot y$ is geodesic in $Y$.
\end{prop}

\begin{proof}
	By the homogeneity of the orbit of $G\cdot y$, it suffices to show that for any $\epsilon,\delta>0$ and any $g\in G$ that does not fix $y$, there is a word $\prod_{j=1}^N g_j=g$ with
	$$\sum_{j=1}^N d_Y(g_jy,y)\le (1+\epsilon)d_Y(gy,y),\quad d_Y(g_jy,y)\le\delta.$$
		
	Let $r_i\to\infty$ such that $(r_i^{-1}\widetilde{M},\tilde{x},\Gamma)$ converges to $(Y,y,G)$ in the equivariant pointed Gromov-Hausdorff sense. We can find $\gamma_i\in \Gamma$ converges to $g$ associated to the above sequence. Let $R_i=d(\gamma_i\tilde{x},\tilde{x})$, then $$R_i/r_i\to d_Y(gy,y)>0.$$ For each $i$, there is $s_i=s(\epsilon,R_i)$ with $R_i/s_i\to\infty$ and a word $\prod_{j=1}^{N_i} \gamma_{i,j}=\gamma_j$ such that $\rho(e,\gamma_{i,j})\le s_i$ for each $i$ and
	$$\sum_{j=1}^{N_i} \rho(e,\gamma_{i,j})\le (1+\epsilon)\rho(e,\gamma_i),$$
	where $s(\epsilon,R)>0$ as in Definition \ref{def_weak_asy_geo}. Since $r_i/s_i\to\infty$, for each $i$ large, we can group successive portions of the word $\prod_{j=1}^{N_i} \gamma_{i,j}$ into a new word $\prod_{j=1}^{K_i} g_{i,j}=\gamma_i$ such that
	$$\delta r_i/2<\rho(e,g_{i,j})<\delta r_i.$$
	By triangle inequality,
	$$\sum_{j=1}^{K_i} \rho(e,g_{i,j})\le \sum_{j=1}^{N_i} \rho(e,\gamma_{i,j})\le (1+\epsilon)\rho(e,\gamma_i).$$
	Combining the above with $\rho(e,g_{i,j})\ge \delta r_i/2$, we have 
	$$K_i\le 2(1+\epsilon)\delta^{-1}r_i^{-1}\rho(e,\gamma_i)\to 2(1+\epsilon)\delta^{-1}d_Y(gy,y).$$
	Passing to a subsequence of $(r_i^{-1}\widetilde{M},\tilde{x},\Gamma)$ if necessary, we can assume that all $K_i$ are the same, denoted as $K$. For each $j=1,...,K$, $\{g_{i,j}\}_i$ sub-converges some element $g_j\in G$, associated to the convergent sequence $(r_i^{-1}\widetilde{M},\tilde{x},\Gamma)\overset{GH}\longrightarrow (Y,y,G)$. Then
	$$d_Y(g_jy,y)=\lim\limits_{i\to\infty} r_i^{-1}\rho(e,g_{i,j}) \le \lim\limits_{i\to\infty }r_i^{-1}\cdot \delta r_i=\delta,$$
	\begin{align*}
	\sum_{j=1}^K d_Y(g_jy,y)&=\lim\limits_{i\to\infty} \sum_{j=1}^K r_i^{-1}\rho(e,g_{i,j})\\
	&\le \lim\limits_{i\to\infty} r_i^{-1} (1+\epsilon)\rho(e,\gamma_i)\\
	&=(1+\epsilon)d_Y(gy,y).
	\end{align*}
    This completes the proof.
\end{proof}

\section{Euclidean orbit in the asymptotic cone}\label{sec_eu}

We prove (3)$\Rightarrow$(4)$\Rightarrow$(1) in Theorem \ref{thm_equivalent} and then Theorem A in this section.

\begin{lem}\label{geo_in_product}
	Let $Y$ be a locally compact length metric space. Let $N$ be a closed subset in the product metric space $\mathbb{R}^k\times Y$, where $\mathbb{R}^k$ is endowed with the standard Euclidean metric. Suppose that $N$ is geodesic in $\mathbb{R}^k\times Y$ and $N$ contains a slice $\mathbb{R}^k\times \{y\}$ for some $y\in Y$. Then $N$ equals $\mathbb{R}^k\times Z$ as a subset of $\mathbb{R}^k\times Y$, where
	 $$Z=N \cap (\{0\}\times Y);$$
	 in particular, $N$ is a product metric.
\end{lem}

\begin{proof}
    We denote $l(t)=(tv,y)$ as an arbitrary unit speed line in the slice $\mathbb{R}^k\times \{y\}$, where $v$ is some unit vector in $\mathbb{R}^k$. Because $N$ is geodesic, for any point $z\in Z$ and any point $l(t)$ on the line, we can draw a unit speed minimal geodesic $\sigma^z_t(s)$ in $N$ from $z$ to $l(t)$. As $t\to\infty$, $\sigma^z_t$ sub-converges to a unit speed ray $\sigma^z_\infty(s)$ starting at $z$ of the form $(sv,z)$ in the product coordinate $\mathbb{R}^k\times Z$. Because $N$ is closed, $\sigma_\infty^z$ must be contained in $N$. Since $v\in\mathbb{R}^k$ and $z\in Z$ are arbitrary, we conclude that $\mathbb{R}^k\times Z\subseteq N$. 
    
    For the other direction $N\subseteq \mathbb{R}^k\times Z$, the argument is similar. Suppose that there is a point $(t_0v,y')$ in $N$ but outside $\mathbb{R}^k\times Z$, where $v$ is of unit length. By drawing minimal geodesics from $(t_0v,y')$ to points on $l(t)=(tv,y)$ with $s<0$ and passing to a convergent subsequence, we obtain a ray $s\mapsto((t_0-s)v,y')$ that is contained in $N$. As this ray passes through $\{0\}\times Y$, we see that $(0,y')\in Z$.
\end{proof}

With Lemma \ref{geo_in_product} above and the Cheeger-Colding splitting theorem, we show that a geodesic orbit $G\cdot y$ must be a product of $\mathbb{R}^k$ and a compact space.

\begin{lem}\label{orbit_eu_cpt}
	Let $(Y,y,G)$ be an equivariant asymptotic cone of $(\widetilde{M},\Gamma)$, where $M$ has $\mathrm{Ric}\ge 0$. Suppose that the orbit $G\cdot y$ is geodesic in $Y$. Then $G\cdot y$ splits isometrically as $\mathbb{R}^k\times Z$, where $Z$ is a compact length metric space.
\end{lem}

\begin{proof}
	If $G\cdot y$ is compact, the conclusion holds trivially. We assume that $G\cdot y$ is non-compact. We pick any point $q\in G\cdot y$ different from $y$. Because $G\cdot y$ is geodesic in $Y$, there is a minimal geodesic $c_q$ that is contained in $G\cdot y$ and connects $y$ to $q$. Let $m$ be the midpoint of $c_q$, and let $h\in G$ such that $hy=m$. Then $h^{-1}\circ c_q$ is a minimal geodesic contained in $G\cdot y$ of midpoint $y$ and length $R=d(y,q)$. For a sequence $q_i\in G\cdot y$ with $R_i=d(y,q_i)\to\infty$, the above process gives a sequence of minimal geodesics in $G\cdot y$ with the same midpoint $y$ and has length $R_i\to\infty$. This sequence of minimal geodesics sub-converges to a line in $G\cdot y$, which is also a line in $Y$. By the Cheeger-Colding splitting theorem \cite{CC1}, $Y$ splits isometrically as $\mathbb{R}\times Y_1$. By Lemma \ref{geo_in_product}, $G\cdot y$ splits isometrically as $\mathbb{R}\times Z_1$ as well. 
	 
	If $Z_1$ is non-compact, we can apply the above argument and Lemma \ref{geo_in_product} again to split $G\cdot y$ isometrically as $\mathbb{R}^2\times Z_2$. Repeating this process, eventually we obtain the desired statement.
\end{proof}

Next we rule out the compact factor $Z$ in the orbit $G\cdot y=\mathbb{R}^k\times Z$. To achieve this, we will consider the set of all the equivariant asymptotic cones of $(\widetilde{M},\Gamma)$ as a whole, denoted as $\Omega(\widetilde{M},\Gamma)$. The following fact is well-known (see \cite[Section 2.1]{Pan_eu} for a proof).

\begin{prop}\label{cpt_cnt}
	Let $(M,x)$ be an open $n$-manifold with $\mathrm{Ric}\ge 0$. Then the set $\Omega(\widetilde{M},\Gamma)$ is compact and connected in the pointed equivariant Gromov-Hausdorff topology.
\end{prop}

We rule out the compact factor by applying a critical scaling argument, which relies on Proposition \ref{cpt_cnt} implicitly. See Remark \ref{rem_crit_scal} after the proof for differences between the argument here and the one in \cite{Pan_eu}, where this kind of argument was first introduced.

\begin{prop}\label{orbit_eu}
	Let $(M,x)$ be an open $n$-manifold of $\mathrm{Ric}\ge 0$. Suppose that for any space $(Y,y,G)\in\Omega(\widetilde{M},\Gamma)$, the orbit $G\cdot y$ is geodesic in $Y$. Then there is an integer $k$ such that $G\cdot y$ is isometric to the standard Euclidean $\mathbb{R}^k$ for any $(Y,y,G)\in\Omega(\widetilde{M},\Gamma)$.
\end{prop}

\begin{proof}
	It suffices to prove that for any space $(Y,y,G)\in\Omega(\widetilde{M},\Gamma)$, the orbit $G\cdot y$ is Euclidean, then the uniformity of the dimension follows from Proposition \ref{cpt_cnt}.
	
	We argue by contradiction. By Lemma \ref{orbit_eu_cpt}, this means that there is an equivariant asymptotic cone $(Y,y,G)\in\Omega(\widetilde{M},\Gamma)$ such that the orbit $G\cdot y$ is isometric to $\mathbb{R}^k\times Z$, where $Z$ is a compact length metric space but not a single point. We apply a blow-down process and a blow-up one to $(Y,y,G)$: for $j\to\infty$,
	$$	(jY,y,G)\overset{GH} \longrightarrow (Y_1,y_1,G_1);\quad
		(j^{-1}Y,y,G)\overset{GH} \longrightarrow (Y_2,y_2,G_2).$$
    It is clear that that $G_2\cdot y_2$ is isometric to $\mathbb{R}^k$, because $G\cdot y$ is isometric to a metric product of $\mathbb{R}^k$ and a compact $Z$. For $(Y_1,y_1,G_1)$, $G_1\cdot y_1$ splits off at least an Euclidean $\mathbb{R}^{k+1}$-factor. In fact, taking a minimal geodesic in $\{0\}\times Z\subset G\cdot y$ starting at $y$, through the scaling $(jY,y,G)$, this segment subconverges to a ray in $G_1\cdot y_1\subset Y_1$. Thus $G_1\cdot y_1$ is isometric to $\mathbb{R}^k\times Z_1$ for some non-compact $Z_1$. It follows from Lemma \ref{orbit_eu_cpt} that $G_1\cdot y_1$ must split off $\mathbb{R}^{k+1}$ isometrically.
    
    By a standard diagonal argument, both spaces $(Y_1,y_1,G_1)$ and $(Y_2,y_2,G_2)$ above belong to $\Omega(\widetilde{M},\Gamma)$. Hence there are sequences $s_i$ and $r_i\to \infty$ such that
    $$(r_i^{-1}\widetilde{M},\tilde{x},\Gamma)\overset{GH}\longrightarrow (Y_1,y_1,G_1),\quad (s_i^{-1}\widetilde{M},\tilde{x},\Gamma)\overset{GH}\longrightarrow (Y_2,y_2,G_2).$$
    Passing to a subsequence, we can assume that $$t_i=s_i^{-1}/r_i^{-1}\to \infty.$$
    We put $(N_i,q_i,\Gamma_i)=(r_i^{-1}\widetilde{M},\tilde{x},\Gamma)$, then
    $$(N_i,q_i,\Gamma_i)\overset{GH}\longrightarrow (Y_1,y_1,G_1),\quad (t_iN_i,q_i,\Gamma_i)\overset{GH}\longrightarrow (Y_2,y_2,G_2).$$
    
    We look for a suitable intermediate rescaling $l_i\to\infty$ and a contradiction in the corresponding limit of $(l_iN_i,q_i,\Gamma_i)$. For each $i$, we consider a set of scales as below:
    \begin{align*}
    L_i:=\{ 1\le l\le t_i\ |\ &d_{GH}((lN_i,q_i,\Gamma),(W,w,H))\le 1/100 \text{ for some space }\\
    & (W,w,H)\in \Omega(\widetilde{M},\Gamma) \text{ such that } H\cdot w \text { is isometric to}\\ 
    &\text{ a standard Euclidean space of dimension at most } k\}.
    \end{align*}
    Recall that $(Y_2,y_2,G_2)$ has an orbit $G_2\cdot y_2$ isometric to $\mathbb{R}^k$, thus $t_i\in L_i$ for all $i$ sufficiently large. We choose $l_i\in L_i$ with $\inf L_i\le l_i\le \inf L_i+1/i$, which are the critical scales.
    
    \textbf{Claim 1:} $l_i\to\infty$. Suppose the contrary, then $l_i$ subconverges to some $l<\infty$. For this subsequence, we have
    $$(l_iN_i,q_i,\Gamma_i)\overset{GH}\longrightarrow (lY_1,y_1,G_1).$$
    Since $l_i\in L_i$, for each $i$ there is a space $(W_i,w_i,H_i)$ with the property described above in the definition of $L_i$. For all $i$ large, we have
    $$d_{GH}((lY_1,y_1,G_1),(W_i,w_i,H_i))\le 1/10.$$   
    On the other hand, recall that in $(Y_1,y_1,G_1)$, the orbit $G_1\cdot y_1$ splits off $\mathbb{R}^{k+1}$ isometrically. After a scaling of $l$, this property holds in $(lY_1,y_1,G_1)$ as well. However, such an orbit in $(lY_1,y_1,G_1)$ cannot be $1/10$-close to $H_i\cdot w_i$ in the pointed Gromov-Hausdorff sense, because $H_i\cdot w_i$ is an Euclidean space of dimension at most $k$. This contradiction leads to claim 1.
    
    As indicated, in the next step we will consider the convergent sequence
    $$(l_iN_i,q_i,\Gamma_i)\overset{GH}\longrightarrow (Y',y',G')$$
    and find a contradiction in the above limit space. We recall that $$(l_iN_i,q_i,\Gamma_i)=(l_ir_i^{-1}\widetilde{M},\tilde{x},\Gamma).$$ Note that $l_ir_i^{-1}\le t_ir_i^{-1}\le s_i^{-1}\to 0$, thus the limit space $(Y',y',G')$ is an equivariant asymptotic cone of $(\widetilde{M},\Gamma)$ as well.
    
    \textbf{Claim 2:} In $(Y',y',G')$, the orbit $G'\cdot y'$ splits off isometrically a Euclidean factor of dimension at most $k$. The argument is very similar to the one that we just applied in claim 1. Because the above convergence and $l_i\in L_i$, there is some space $(W,w,H)$ with the property in the definition of $L_i$ such that
    $$d_{GH}((Y',y',G'),(W,w,H))\le 1/10.$$
    If the orbit $G'\cdot y'$ splits off a Euclidean factor of dimension $k+1$, then $G'\cdot y'$ cannot be close to $H\cdot w$, which is a Euclidean space of dimension at most $k$. This proves claim 2.
    
    By Lemma \ref{orbit_eu_cpt}, the orbit $G'\cdot y'$ is isometric to a product metric $\mathbb{R}^m\times Z'$, where $Z'$ is a compact length metric space. It follows from claim 2 that $m\le k$. 
    
    If $Z'$ is a single point, we consider the sequence
    $$((l_i/2)N_i,q_i,\Gamma_i)\overset{GH}\longrightarrow (2^{-1}Y',y',G').$$
    In the new limit space $(2^{-1}Y',y',G')$, $G'\cdot y'$ is isometric to $\mathbb{R}^m$ as well. Since $m\le k$, this shows that $l_i/2\in L_i$ for $i$ sufficiently large, which contradicts our choice of $l_i$ being close to $\inf L_i$.
    
    If $Z'$ is not a single point, we make use of an asymptotic cone of $(Y',y',G')$:
    $$(j^{-1}Y',y',G')\overset{GH}\longrightarrow (W',w',H'),$$
    where $j\to\infty$. Because the orbit $G'\cdot y'$ is isometric to $\mathbb{R}^m\times Z'$ and $Z'$ is compact, in $(W',w',H')$ we have $H'\cdot w'$ being isometric to $\mathbb{R}^m$. We choose a large $J$ such that
    $$d_{GH}((J^{-1}Y',y',G'),(W',w',H'))\le 1/10^3.$$
    Together with
    $$((l_i/J)N_i,q_i,\Gamma_i)\overset{GH}\longrightarrow (J^{-1}Y',y',G'),$$
    we see that
    $$d_{GH}(((l_i/J)N_i,q_i,\Gamma_i),(W',w',H'))\le 1/100$$
    for all $i$ sufficiently large. It follows that $l_i/J\in L_i$ for $i$ large and we end in a contradiction with the choice of $l_i$ again.
    
    We have ruled out all the possibilities of $(Y',y',G')$. This completes the proof.
\end{proof}

\begin{rem}\label{rem_crit_scal}
	We compare the critical scaling argument above to the one in \cite{Pan_eu}. Besides other details, one major difference is when arranging the order of a sequence and its rescaling sequence. Here we have the sequences
	$$(N_i,q_i,\Gamma_i)\overset{GH}\longrightarrow (Y_1,y_1,G_1),\quad (t_iN_i,q_i,\Gamma_i)\overset{GH}\longrightarrow (Y_2,y_2,G_2)$$
	arranged so that $G_1\cdot y_1$ being strictly larger than $G_2\cdot y_2$; while in \cite{Pan_eu}, $G_2$ is strictly larger than $G_1$ (see proof of Proposition 3.5 in \cite{Pan_eu} for instance). Note that both limit spaces are in $\Omega(\widetilde{M},\Gamma)$, so one can arrange any one of the space to be the limit of the $(N_i,q_i,\Gamma_i)$ and the other one to be the limit of the rescaling sequence. If one change the order in these proofs, then both arguments actually would not end in the desired contradiction. 
	
	There are reasons behind this. We explain the geometric intuitions involved. Since $\Omega(\widetilde{M},\Gamma)$ is connected, one may think of an $\epsilon$-chain connecting two spaces $(Y_1,y_1,G_1)$ and $(Y_2,y_2,G_2)$; moreover, each element in the chain is $\epsilon$-close to a blow-up (scale by a number that is slightly larger than $1$) of the previous one. 
	
	Here, we have the assumption that each space in $\Omega(\widetilde{M},\Gamma)$ has a geodesic orbit at the base point. If one arranges $G_2\cdot y_2=\mathbb{R}^{k}$ as the limit of $(N_i,q_i,\Gamma_i)$ and $G_1\cdot y_1=\mathbb{R}^{k+1}$ as the limit of $(t_iN_i,q_i,\Gamma_i)$, then a possible $\epsilon$ chain involving a gradual blow-up process could exist as described below. First, grow a small $S^1$ factor on $\mathbb{R}^k$, then this small factor becomes larger and larger, and eventually, we see $\mathbb{R}^{k+1}$. The existence of such a chain indicates that if we arrange the sequences in this way, then we would not end in a contradiction. If one arranges the sequence as in the proof, one cannot think of a chain with a gradual blow-up process that changes the orbit from $\mathbb{R}^{k+1}$ to $\mathbb{R}^k$ while assuming all the spaces in this chain have geodesic orbits.
	
	In \cite{Pan_eu}, taking proof of Proposition 3.5 for instance, we have the property that small groups cannot grow out of nowhere (\cite[Lemma 3.3]{Pan_eu}). This eliminates the existence of a gradual blow-up process that changes from a trivial action to a non-trivial one. Other critical scaling arguments in \cite{Pan_eu} have a similar, though more complicated, logic behind the arrangements.
\end{rem}

Next, we complete the proof of Theorem \ref{thm_equivalent}.

\begin{proof}[Proof of Theorem \ref{thm_equivalent}]
	We have shown (1)$\Rightarrow$(2) as Proposition \ref{asy_geo}, (2)$\Rightarrow$(3) as Proposition \ref{orbit_geo}, and (3)$\Rightarrow$(4) as Proposition \ref{orbit_eu}. It remains to prove that (4)$\Rightarrow$(1).
	
	Suppose that $E(M,x)=\delta>0$. By the definition of $E(M,x)$, this means that we can find a sequence $\gamma_i\in \Gamma$ with the properties below:\\
	(i) $r_i=d(\gamma_i\tilde{x},\tilde{x})\to \infty$,\\
	(ii) there is a minimal geodesic loop $c_i$ based at $x$ representing $\gamma_i$, such that $c_i$ is not contained in $B_{\delta r_i/2}(x)$.
	
	Lifting $c_i$ to $(\widetilde{M},\tilde{x})$, we obtain $\tilde{c}_i$ as a minimal geodesic from $\tilde{x}$ to $\gamma\tilde{x}$. By the property (ii) above, $\tilde{c}_i$ is not contained in $\pi^{-1}(B_{\delta r_i/2}(x))$, where $\pi:\widetilde{M}\to M$ is the covering map. We consider the convergence
	\begin{center}
		$\begin{CD}
		(r^{-1}_i\widetilde{M},\tilde{x},\Gamma,\gamma_i) @>GH>> 
		(Y,y,G,g)\\
		@VV\pi V @VV\pi V\\
		(r^{-1}_iM,x) @>GH>> (Z=Y/G,z).
		\end{CD}$
	\end{center}
	It is clear that $d(gy,y)=1$. By assumption, $G\cdot y$ is a standard Euclidean space. With respect to this convergent sequence, $\tilde{c}_i$ subconverges to a minimal geodesic $\sigma$ from $y$ to $gy$. Moreover, $\sigma$ is not contained in $\pi^{-1}(B_{1/3}(z))$. On the other hand, by hypothesis, the unique minimal geodesic from $y$ to $gy$ is contained in $G\cdot y=\pi^{-1}(z)$. We result in a contradiction. Therefore, $E(M,x)=0$.
\end{proof}

\begin{cor}\label{cor_zero}
	Let $(M,x)$ be an open manifold of $\mathrm{Ric}\ge 0$. If $E(M,x)=0$, then $E(M,y)=0$ for all $y\in M$.
\end{cor}

\begin{proof}
    We write $\Gamma_x$ and $\Gamma_y$ as $\pi_1(M,x)$ and $\pi_1(M,y)$ respectively. Let $\pi:\widetilde{M}\to M$ be the covering map and let $\tilde{x}\in \pi^{-1}(x)$, $\tilde{y}\in \pi^{-1}(y)$. Note that both the orbits $\Gamma_x\cdot\tilde{x}$ and $\Gamma_y\cdot \tilde{x}$ are identified as $\pi^{-1}(x)$. For any sequence $r_i\to \infty$, $\Gamma_x\cdot\tilde{x}=\Gamma_y\cdot \tilde{x}$ converges to a geodesic subset in the asymptotic cone $(Y,y)$ of $\widetilde{M}$. Because $\Gamma_y\cdot \tilde{y}$ converges to the same limit orbit, we see that statement (3) of Theorem \ref{thm_equivalent} holds for $(\widetilde{M},\Gamma_y)$. It follows from Theorem \ref{thm_equivalent} that $E(M,y)=0$.
\end{proof}

In \cite{Pan_eu,Pan_al_stable}, the author proved that if $\widetilde{M}$ satisfies certain stability condition at infinity (see \cite[Definition 1.1]{Pan_eu} and \cite[Definition 0.1]{Pan_al_stable}), then the limit orbit in the asymptotic cone of $\widetilde{M}$ that comes from a nilpotent group action is always Euclidean. Together with Theorem \ref{thm_equivalent}, we conclude the following:

\begin{cor}\label{cor_cone_zero}
	Let $(M,x)$ be an open manifold of $\mathrm{Ric}\ge 0$. Suppose that its Riemannian universal cover $\widetilde{M}$ is $k$-Euclidean at infinity or $(C(X),\epsilon_X)$-stable at infinity, then $E(M,x)=0$.
\end{cor}

\begin{proof}
	Let $N$ be a nilpotent subgroup of $\Gamma=\pi_1(M,x)$ with finite index. In \cite{Pan_eu,Pan_al_stable}, it was shown that under either assumption, for any $r_i\to\infty$, the equivariant asymptotic cone of $(\widetilde{M},N)$
	$$(r_i^{-1}\widetilde{M},\tilde{x},N)\overset{GH}\longrightarrow (Y,y,H)$$ 
	satisfies that $H\cdot y$ is a Euclidean factor $\mathbb{R}^k$ in $Y$. For the corresponding equivariant asymptotic cone of $(\widetilde{M},\Gamma)$:
	$$(r_i^{-1}\widetilde{M},\tilde{x},\Gamma)\overset{GH}\longrightarrow (Y,y,G),$$
	since $N$ has finite index in $\Gamma$, it follows that 
	$$G\cdot y=G_0\cdot y=H\cdot y,$$
	where $G_0$ is the identity component subgroup of $G$. By Theorem \ref{thm_equivalent}, we conclude that $E(M,x)=0$.
\end{proof}

We conjecture that Corollary \ref{cor_zero} can be extended to manifolds whose universal covers have Euclidean volume growth. Also see \cite[Conjecture 0.2]{Pan_al_stable}.

\begin{conj}
	Let $(M,x)$ be an open manifold of $\mathrm{Ric}\ge 0$. If its Riemannian universal cover has Euclidean volume growth, then $E(M,x)=0$. 
\end{conj}

We prove Theorem A in the remainder of this section. We recall the following lemma about elements in a nilpotent subgroup of $\mathrm{Isom}(\mathbb{R}^l)$. See \cite[Lemma 2.4]{Pan_al_stable} for a proof.

\begin{lem}\label{commute_E}
	Let $G$ be a nilpotent subgroup of $\mathrm{Isom}(\mathbb{R}^l)$. Let $(A,x)$ and $(B,y)$ be two elements of $G$. Then $(A,x)$ and $(B,y)$ commute if and only if $A$ and $B$ commute.
\end{lem}

We derive a result on the transitive nilpotent group actions on $\mathbb{R}^l$.

\begin{lem}\label{pre_nil_trans}
   Let $H$ be a closed nilpotent subgroup of $\mathrm{Isom}(\mathbb{R}^l)$. Suppose that $H$-action is transitive on $\mathbb{R}^l$, then $H=\mathbb{R}^l$ acts as translations.
\end{lem}

\begin{proof}
	We first prove the statement when $H$ is a closed abelian subgroup. We know that $H$ is isomorphic to $\mathbb{R}^{m}\times T\times F$, where $m$ is an integer, $T$ is a torus, and $F$ is a discrete abelian group. Because $H$ acts transitively on $\mathbb{R}^l$, it follows that $m=l$ and $H/\mathrm{Iso}_0(H)$ is diffeomorphic to $\mathbb{R}^l$, where $\mathrm{Iso}_0(H)$ is the isotropy subgroup at $0\in\mathbb{R}^l$. We claim that $\mathrm{Iso}_0(H)$ is trivial. Recall that any isometry of $\mathbb{R}^l$ can be written as $(A,v)$ for some $A\in O(l)$ and $v\in\mathbb{R}^l$. Let $(A,0)$ be an element of $\mathrm{Iso}_0(H)$. Since $H$ acts transitively, for any $v\in \mathbb{R}^l$, there is $(B,v)\in H$ that moves $0$ to $v$. By direct calculation, the fact that $(A,0)$ and $(B,v)$ commute implies that $Av=v$. Since $v$ is arbitrary in $\mathbb{R}^l$, we conclude that $A=I$. This verifies the statement when $H$ is abelian.
	
	For the nilpotent case, let 
	\begin{align*}
	p:\mathrm{Isom}(\mathbb{R}^l)&\to O(l)\\
	(A,v)&\mapsto A
	\end{align*}
    be the projection map. The closure of $p(H)$ is a compact nilpotent Lie group of $O(l)$. Hence its identity component subgroup, denoted as $K$, must be abelian. Put $H'=p^{-1}(K)$. By Lemma \ref{commute_E}, $H'$ is an abelian subgroup of $H$ with finite index. As we have shown in the first paragraph, it follows that $H'$ acts on $\mathbb{R}^l$ by translations; in other words, elements of $H'$ are of the form $(I,w)$, where $w\in\mathbb{R}^k$. For any element $(A,v)$ in $H-H'$, again according to Lemma \ref{commute_E}, it commutes with any element $(I,w)\in H'$. A direct calculation gives $Aw=w$. Because $w$ is arbitrary, we see that $A=I$ and thus $(A,v)\in H'$. This shows that $H=H'$ and the result follows.
\end{proof}

\begin{cor}\label{nil_trans}
	Let $(M,x)$ be an open manifold of $\mathrm{Ric}\ge 0$ and $E(M,x)=0$. Let $N$ be a nilpotent subgroup of $\pi_1(M,x)$ with finite index. Then for any equivariant asymptotic cone $(Y,y,H)$ of $(\widetilde{M},\tilde{x},N)$, the following holds:\\
	(1) $H\cdot y$ is isometric to a standard Euclidean space $\mathbb{R}^l$;\\
	(2) $H$ acts on $H\cdot y$ as translations.
\end{cor}

\begin{proof}
	(1) For a sequence $r_i\to\infty$, we have
	$$(r_i^{-1}\widetilde{M},\tilde{x},\Gamma,N)\overset{GH}\longrightarrow (Y,y,G,H).$$
	Since $N$ has finite index in $\Gamma$ and the orbit $G\cdot y$ is isometric to $\mathbb{R}^l$ (in particular, it is connected), it follows that
	$$G\cdot y=G_0\cdot y=H\cdot y.$$
	
	(2) Note that $H$ is nilpotent. The result follows directly from (1) and Lemma \ref{pre_nil_trans}.
\end{proof}

\begin{proof}[Proof of Theorem A]
	With Corollary \ref{nil_trans} in hand, the remaining proof is essentially Section 4 of \cite{Pan_al_stable}. We give a sketch here, with all the details can be found in \cite{Pan_al_stable}.
	
	By \cite{Mil,Gro_poly}, $\Gamma$ contains a nilpotent subgroup $N$ of finite index. To prove that $\Gamma$ is almost abelian, it suffices to find an abelian subgroup of $N$ with finite index.
	
	The first step is showing that 
	$$|\gamma^2|\ge 1.9\cdot |\gamma|$$
	holds for all $\gamma\in N$ with large displacement, where $|\gamma|=d(\gamma\tilde{x},\tilde{x})$; in other words, $\gamma$ acts as an almost translation at $\tilde{x}$. Suppose the statement is not true, then we can find a contradicting sequence $\gamma_i\in N$ with $|\gamma_i^2|<1.9\cdot |\gamma_i|$. Put $r_i=|\gamma_i|\to\infty$ and consider the convergent sequence
	$$(r_i^{-1}\widetilde{M},\tilde{x},N,\gamma_i)\overset{GH}\longrightarrow (Y,y,H,h).$$
	A contradiction would arise here since $h$ satisfies
	$d(h^2y,y)\le 1.9 d(hy,y)$ but Corollary \ref{nil_trans} has $h$ acting as a translation at $y$.
	
	Next, for a left-invariant distance function on a nilpotent group $N$ with the above almost translation property for elements with large displacement, it can be shown inductively that $[N,N]$ must be a finite group (see proof of Lemma 4.7 in \cite{Pan_al_stable}). It follows from a fact from group theory that $Z(N)$, the center of $N$, must have finite index in $N$ (see proof of Theorem 4.1 in \cite{Pan_al_stable}).
\end{proof}

\begin{rem}
	We can start with a normal nilpotent subgroup $N$ of $\Gamma$ in the proof. Then the abelian subgroup $Z(N)$ obtained above is normal in $\Gamma$.
\end{rem}

We end the main part of this paper with a conjecture on the escape rate gap (also see Remark \ref{rem_nil_gap}).
\begin{conj}\label{quest_gap}
	Given $n$, there is a universal constant $\epsilon(n)>0$ such that for any open $n$-manifold $(M,x)$ of $\mathrm{Ric}\ge 0$, if $E(M,x)\le \epsilon(n)$, then $\pi_1(M,x)$ is virtually abelian.
\end{conj}


\appendix

\section{The bounded case}\label{app_bdd}

As mentioned in the introduction, it follows from the Cheeger-Gromoll splitting theorem that if $\sup_{\gamma\in \Gamma} d_H(x,c_\gamma)$ is finite, then $\pi_1(M,x)$ is virtually abelian. This result is well-known to experts, but we cannot find a proof in the literature, so we include a proof here for readers' convenience.

\begin{prop}\label{bounded_case}
	Let $(M,x)$ be an open manifold of $\mathrm{Ric}\ge 0$. If $\sup_{\gamma\in \Gamma} d_H(x,c_\gamma)$ is finite, then $\pi_1(M,x)$ is virtually abelian.
\end{prop}

\begin{lem}\label{app_nil_to_abel}
	Let $\Gamma$ be a nilpotent subgroup of $K\times \mathrm{Isom}(\mathbb{R}^k)$, where $K$ is a compact Lie group. Then $\Gamma$ is virtually abelian. 
\end{lem}

\begin{proof}
	We first prove the case that $K$ is trivial. We consider the group homomorphism
	\begin{align*}
		\pi: \mathrm{Isom}(\mathbb{R}^k)&\to O(n)\\
		(A,x)&\mapsto A
	\end{align*}
	Let $H$ be the closure of $\pi(\Gamma)$ in $O(k)$. $H$ is a compact nilpotent Lie group. Its identity component $H_0$ has finite index in $H$. Moreover, because $H_0$ is connected and nilpotent, it must be a torus. Therefore, by Lemma \ref{commute_E}, $\pi^{-1}(H_0)$ is an abelian subgroup of finite index in $\Gamma$.
	
	For the general case, let
	$$p_1: \Gamma\to K, \quad p_2: \Gamma\to \mathrm{Isom}(\mathbb{R}^k)$$
	be the natural projections. We have shown that $p_2(\Gamma)$ has an abelian subgroup $A_2$ of finite index. For $p_1(\Gamma)$, its closure is a compact nilpotent subgroup of $K$. By taking its identity component as in the first paragraph, it follows that $p_1(\Gamma)$ has an abelian subgroup $A_1$ of finite index. Then $\Gamma\cap (A_1\times A_2)$ is the desired abelian subgroup of finite index in $\Gamma$.
\end{proof}

\begin{proof}[Proof of Proposition \ref{bounded_case}]
	From the Cheeger-Gromoll splitting theorem \cite{CG_split}, $\widetilde{M}$ splits isometrically as $\mathbb{R}^k\times Z$, where $Z$ does not contain any line. Since isometries take lines to lines, the isometry group of $\widetilde{M}$ splits as
	$$\mathrm{Isom}(\widetilde{M})=\mathrm{Isom}(\mathbb{R}^k)\times \mathrm{Isom}(Z).$$
	Let 
    $ q:\mathrm{Isom}(\widetilde{M})\to \mathrm{Isom}(Z)$
	be the natural projection. We write $\tilde{x}=(0,z)$.
	
	We claim that
	$\{q(\gamma)\cdot z|\gamma\in \Gamma\}$
	is bounded in $Z$. Suppose the contrary, then there is a sequence $\gamma_i\in \Gamma$ such that $d_Z(z,q(\gamma_i)\cdot z)\to \infty$. For each $i$, let $\tilde{c}_i:[0,l_i]\to\widetilde{M}$ be a unit speed minimal geodesic from $\tilde{x}$ to $\gamma_i\tilde{x}$. Let $\pi:\widetilde{M}\to M$ be the covering map. By hypothesis, $\tilde{c}_i$ is contained in $\pi^{-1}(\overline{B_R}(\tilde{x}))$ for all $i$, where $$R=\sup_{\gamma\in \Gamma} d_H(x,c_\gamma)<\infty.$$ 
	Note that $\pi^{-1}(\overline{B_R}({x}))$ is $\Gamma$-invariant with a compact quotient $\overline{B_R}(x)$. Hence for each midpoint $m_i=c_i(l_i/2)$, there is $\beta_i\in \Gamma$ such that $d(m_i,\beta_i\tilde{x})\le R$. $\beta_i^{-1}\circ c_i$ is a sequence of minimal geodesics with whose midpoints are at most $R$ away from $\tilde{x}$. Projecting these minimal geodesics $\beta_i^{-1}\circ c_i$ to the $Z$-factor, we obtain a sequence of minimal geodesics with length equal to $d_Z(z,q(\gamma_i)\cdot z)\to \infty$ and midpoints being at most $R$ away from $z$. Passing to a subsequence, we result in a line of $Z$. A contradiction. This verifies the claim.
	
	The claim implies that $q(\Gamma)$ is pre-compact in $\mathrm{Isom}(Z)$. Let $K$ be its closure in $\mathrm{Isom}(Z)$. By \cite{Mil,Gro_poly}, $\Gamma$ contains a nilpotent subgroup $N$ of finite index. Then $N$ is a subgroup of $\mathrm{Isom}(\mathbb{R}^k)\times K$. The result follows directly from Lemma \ref{app_nil_to_abel}.
\end{proof}

\section{Examples and their escape rates}\label{app_exmp}

We estimate the escape rates of some known and new examples of open manifolds as warped products $[0,\infty)\times_f S^{p-1}\times N_r$ with $\mathrm{Ric}>0$. We do not seek the most general statements or the most effective estimates here.

\begin{exmp}(Nabonnand/Bergery's examples) Nabonnand constructed a doubly warped product $M^4=[0,\infty)\times_f S^2 \times_h S^1$ of positive Ricci curvature \cite{Nab}. Later, Bergery generalized Nabonnand’s method and constructed a doubly warped product $M=[0,\infty)\times_f S^{p-1} \times_h N$ of $\mathrm{Ric}_M>0$ for any closed manifold $N$ of $\mathrm{Ric}_N\ge 0$ \cite{Bergery}. Note that in these examples, because $M$ is diffeomorphic to $\mathbb{R}^{p}\times N$ and $N$ admits a metric of nonnegative Ricci curvature, $\pi_1(M)=\pi_1(N)$ is virtually abelian.
	
	We write the doubly warped product on $M=[0,\infty)\times_f S^{p-1}\times_h N$ as
	$$g=dr^2+f(r)^2ds^2+h(r)^2g_0,$$
	where $ds^2$ is the standard metric on the unit sphere $S^{p-1}$ and $g_0$ is a metric on $N$. At $r=0$, $f$ and $h$ satisfies
	$$f(0)=0,\quad f'(0)=1,\quad f''(0)=0,\quad h(0)>0,\quad h'(0)=0.$$
    In these constructions, $h$ is strictly decreasing and satisfies $h=(f')^{1/q}$, where $q=\dim(N)$. The function $f$ is not explicit, as the solution to a differential equation. $f$ has $f'>0$ but also the freedom to satisfy $f'\to 0$ or $f'\to c>0$ as $r\to\infty$ (see \cite{Bergery}). When $f'\to c>0$, the diameter of $N$ converges to some positive constant as $r\to\infty$.
\end{exmp}

\begin{prop}\label{end_c}
	Let $M=[0,\infty)\times_f S^{p-1} \times_h S^1$ be a doubly warped product as above. Fix a reference point $y\in S^1$. Let $x=(0,y)\in M$, where $0\in\mathbb{R}^p$ and $y\in S^1$; let $S^1(r)$ be a copy of $S^1$ at distance $r$. Suppose that $\mathrm{diam}(S^1(r))$ is strictly decreasing with a positive limit, then $E(M,x)=0$.
\end{prop}

\begin{proof}
	Let $\delta(r)$ be the length of the circle $S^1(r)$ with metric $h(r)^2g_0$. By assumption, $\delta(r)=c+\epsilon(r)$ for some strictly decreasing function $\epsilon(r)\to 0$ and some $c>0$. For each positive integer $l$, we choose $r_l>0$ such that
	$$2r_l=l\cdot \epsilon(r_l).$$
	Note that $r_l/l\to 0$.
	Let $\sigma_l$ be a loop based at $x$ constructed as below: first go along a minimal geodesic from $x$ to $y\in S^1(r_l)$, then go around $S^1(r_l)$ $l$ times, then go back to $x$ along a minimal geodesic. By construction, 
	$$\mathrm{length}(\sigma_l)=lc+l\epsilon(r_l)+2r_l=lc+4r_l.$$
	Let $c_l$ be a minimal geodesic loop at $x$ representing the class of $\sigma_l$ in $\pi_1(M,x)$ and let $R_l=d_H(x,c_l)$. 
	
	We define a map $F$ from $M$ to a cylinder $[0,\infty)\times S^1$ with a product metric, where the metric on the $S^1$-factor has length $c$, as follows:
	For any point in $M$, we can write it as $(v,z)$, where $v\in\mathbb{R}^p$
    and $z\in S^1$. We denote $|v|$ as the distance between $0$ and $v$ in $\mathbb{R}^p$ with metric $[0,\infty)\times_f S^{p-1}$. We define $F$ as
	$$F:M\to [0,\infty)\times S^1,\quad (v,z)\to (|v|,z).$$
	It is clear that $F$ is a Lipschitz map. $F(c_l)$ is a loop contained exactly in $[0,R_l]\times S^1$. As a loop based at a point in $\{0\}\times S^1$ that touches $\{R_l\}\times S^1$ and also warps the circle $l$ times, we can give a lower bound on the length of $F(c_l)$:
	$$\mathrm{length}(c_l)\ge \mathrm{length}(F(c_l))\ge 2\left(R_l^2+\frac{1}{4}c^2l^2\right)^{1/2}.$$
	Consequently, $\mathrm{length}(c_l)\le \mathrm{length}(\sigma_l)$ yields
	$$2\left(R_l^2+\frac{1}{4}c^2l^2\right)^{1/2}\le lc+4r_l.$$
	Hence $R_l^2\le 2lcr_l+4r_l^2$ and
	$$\dfrac{d_H(x,c_l)}{\mathrm{length}(c_l)}\le \dfrac{R_l}{lc}\le \left(\dfrac{2r_l}{lc}+\dfrac{4r_l^2}{l^2c^2}\right)^{1/2}\to 0.$$
	We conclude that $E(M,x)=0$
\end{proof}

\begin{rem}
	Proposition \ref{end_c} also holds if one replaces the circle $S^1$ by a closed manifold $N$ with a metric $g_0$ of nonnegative Ricci curvature.
\end{rem}

\begin{exmp}\label{exmp_Wei}(Wei's examples)
Wei constructed an open manifold $M$ of $\mathrm{Ric}>0$ with a torsion-free nilpotent non-abelian fundamental group \cite{Wei}. We briefly recall this construction as below. Let $\widetilde{N}$ be a simply connected nilpotent Lie group. Note that $\widetilde{N}$ does not admit a metric of nonnegative Ricci curvature, but it admits one of almost nonnegative sectional curvature. Let $\{X_1,...,X_n\}$ be a triangular basis of the Lie algebra of $\widetilde{N}$, that is, $[X_j,X]\in l_{j-1}$ for all $X\in \mathrm{Lie}(\widetilde{N})$, where $l_j$ is spanned by $X_1$,..,$X_{j-1}$. We define a family of left-invariant metrics $\widetilde{g_r}$ on $\widetilde{N}$ by setting $\{X_i\}$ to be orthogonal and 
$$||X_i||_r=h_i(r)=(1+r^2)^{-\alpha_i},$$
where $\alpha_n=\alpha>0$ and $2\alpha_i-4\alpha_{i+1}=1$ for $1\le i\le n-1$.
Let $\Gamma$ be a lattice of $\widetilde{N}$, then $N=\widetilde{N}/\Gamma$ endows a family of metrics $g_r$ such that
$$\mathrm{Ric}(g_r)\ge -\dfrac{c}{1+r^2},$$
where $c>0$. Denote $(N,g_r)$ as $N_r$. Then $M=[0,\infty)\times_f S^{p-1}\times N_r$ with the metric
$$g=dr^2+f(r)^2 ds^2+g_r,$$
has positive Ricci curvature, where $f(r)=r(1+r^2)^{-1/4}$ and $p$ is sufficiently large. Note that this metric is not a doubly warped product in the usual sense since the metrics on $N_r$ decay at different rates for different steps of $N$.
When $\dim(N)>1$, $\pi_1(M)=\Gamma$ is not virtually abelian.
\end{exmp}

By Theorem A, the above example ought to have a positive escape rate. Here we estimate a positive lower bound of $E(M,x)$ for $M=[0,\infty)\times_f S^{p-1}\times_h S^1$, where $x=(0,y)\in M$. The nilpotent example $N$ clearly shares the same lower bound. 

Let $\delta(r)$ be the length of $S^1(r)$. Scaling by a constant if necessary, we write $\delta(r)=(1+r^2)^{-\alpha}$, where $\alpha>0$. For each integer $l>0$, we choose $r_l$ such that $2r_l=l\delta(r_l)$. Then $r_l\to \infty$ as $l\to\infty$. We consider a loop $\sigma_l$ similarly as in the proof of Proposition \ref{end_c}: first go from $x$ to $y\in S^1(r_l)$ along a minimal geodesic, go around $S^1(r_l)$ $l$ times, then go back to $x$. $\sigma_l$ has length 
$$\mathrm{length}(\sigma_l)=2r_l+l\cdot \delta(r_l)=4r_l=2l\delta(r_l).$$
Let $c_l$ be a minimal geodesic loop based at $x$ representing the class of $\sigma_l$. Then
$$\mathrm{length}(c_l)\le \mathrm{length}(\sigma_l).$$
Let $R_l=d_H(x,c_l)$. Sine $c_l$ is contained in $\overline{B_{R_l}}(x)$, then $R_l$
satisfies 
$$l\cdot \delta(R_l)\le \mathrm{length}(c_l)\le 2l \delta(r_l),$$
that is,
$$(1+R_l^2)^{-\alpha}\le 2(1+r_l^2)^{-\alpha}.$$
Hence
$$\limsup_{l\to\infty}\dfrac{R_l}{r_l}\ge \left(\dfrac{1}{2}\right)^{\frac{1}{2\alpha}}.$$
It follows that
$$E(M,x)=\limsup_{l\to\infty} \dfrac{d_H(x,c_l)}{\mathrm{length}(c_l)}\ge \limsup_{l\to\infty} \dfrac{R_l}{4r_l}\ge \left(\dfrac{1}{2}\right)^{2+\frac{1}{2\alpha}}>0.$$

\begin{rem}
	Since the above example has a fundamental group $\mathbb{Z}$ and $E(M,x)>0$, the converse of Theorem A is not true in general.
\end{rem}

\begin{rem}
	With a better choice of $\sigma_l$, one can show that
	$$E(M,x)\ge \dfrac{1}{\alpha^{1/\alpha}(2+1/\alpha)}>0;$$
	in particular, $E(M,x)\to 1/2$ as $\alpha\to \infty$.
\end{rem}

\begin{rem}
	By a similar, but slightly more involved, argument as the proof of Proposition \ref{end_c}, one can show that Wei's construction $[0,\infty)\times_f S^{p-1}\times_h S^1$ has $E(M,x)\le b(\alpha)$, where $b(\alpha)\to 0$ as $\alpha\to 0$.
\end{rem}

\begin{rem}\label{rem_nil_gap}
	If $N$ is not a circle, for example, $N$ is the Heisenberg $3$-manifold, then we can use 
	$$h_2(r)=h_3(r)=(1+r^2)^{-\alpha},\quad h_1(r)=(1+r^2)^{-\beta},$$ where $\alpha>0$ and $\beta=2\alpha+1/2$
    as given in Example \ref{exmp_Wei}. This example actually has a uniform positive bound for $E(M,x)$, regardless of the choice of $\alpha>0$. We are not aware of examples with $\mathrm{Ric}\ge 0$, torsion-free nilpotent non-abelian fundamental groups, and arbitrarily small escape rates. This leads to Conjecture \ref{quest_gap}. 
\end{rem}

\begin{exmp}(An example with a logarithm decay $h(r)$) We construct an open manifold $M=[0,\infty)\times_f S^{p-1}\times_h S^1$ of $\mathrm{Ric}>0$ and $E(M,x)=0$. Unlike the case in Proposition \ref{end_c}, here the diameter of $S^1(r)$ tends to $0$ as $r\to\infty$. We use the warping functions
$$f(r)=\dfrac{\sqrt{\ln 2} \cdot r }{\ln^{1/2}(2+r^2) },\quad h(r)=\dfrac{1}{\ln^\alpha(2+r^2)},$$
where $\alpha>0$. It is straight-forward to check that
$$f(0)=0,\quad f'(0)=1,\quad f''(0)=0.$$
$$f'>0,\quad 1-(f')^2\ge 0,\quad f''\le 0,\quad  h'<0.$$
Let $H=\partial/\partial r$, $U$ a unit vector tangent to $S^{p-1}$, and $X$ a unit vector tangent to $S^1$. By direct calculation,
\begin{align*}
	&\mathrm{Ric}(H,H)=-\dfrac{h''}{h}-(p-1)\dfrac{f''}{f},\\
	&\mathrm{Ric}(U,U)=-\dfrac{f''}{f}+\dfrac{p-2}{f^2} \left[1-(f')^2\right]-\dfrac{f'h'}{fh},\\
	&\mathrm{Ric}(X,X)=-\dfrac{h''}{h}-(p-1)\dfrac{f'h'}{fh}.
\end{align*}
It is straight-forward to check that with the above $f$ and $h$, $M$ has positive Ricci curvature when $p$ is sufficiently large.
\end{exmp}

Unlike Wei's construction, one cannot use this type of function to construct a warped product $[0,\infty)\times_f S^{p-1}\times N_r$ of positive Ricci curvature, where $N=\widetilde{N}/\Gamma$ is a compact nilpotent manifold. We explain the obstructions. $\widetilde{N}$ does admit a family of metrics $\widetilde{g_r}$ of almost nonnegative Ricci curvature by using
$$||X_i||_r=h_i(r)=\dfrac{1}{\ln^{\alpha_i}(2+r^2)},$$
where $\alpha_i>0$ as in Wei's construction. Then $(N,g_r)$ satisfies
$$\mathrm{Ric}(g_r)\ge -\dfrac{c}{\ln(2+r^2)},$$
where $c>0$. However, the warped product has
$$\mathrm{Ric}(X_i,X_i)=-\dfrac{h_i''}{h_i}-(p-1)\dfrac{f'h_i'}{fh_i}+\mathrm{Ric}(g_r)(X_i)-\sum_{j\not= i} \dfrac{h'_ih'_j}{h_ih_j};$$
$\mathrm{Ric}(g_r)(X_i)$ is the dominating term, which decays at a rate of $\ln^{-1}(r)$, while all other terms decays at a rate of $r^{-2}$. Hence it is impossible to obtain positive Ricci curvature by raising $p$. 

This matches our Theorem A. Actually, $E(M,x)=0$ as we show below for this type of warping functions $h$.

Let $\delta(r)=\ln^{-\alpha}(2+r^2)$ be the length of $S^1(r)$, where $S^1(r)$ is a copy of $S^1$ in $M=[0,\infty)\times_f S^{p-1}\times_h S^1$ at distance $r$. For each integer $l>0$, we choose $r_l>0$ such that the function 
$$r\mapsto 2r+l\cdot\delta(r)$$
obtains the minimum at $r_l$; in other words, $r_l$ such that
$$\alpha lr_l=\ln^{\alpha+1}(2+r_l^2)\cdot (2+r_l^2).$$
Let $\sigma_l$ be the loop constructed as previously and let $c_l$ be a minimal geodesic loop representing the class of $\sigma_l$. We have 
$$\mathrm{length}(\sigma_l)\le 2r_l+\delta(r_l)l.$$
Let $R_l=d_H(x,c_l)$. To estimate a lower bound of the length of $c_l$, we consider a cylinder $[0,\infty)\times S^1$ and a map 
$$F:M\to [0,\infty)\times S^1, (v,z)\mapsto (|v|,z)$$
as we did in the proof of Proposition \ref{end_c}. Here we endow the circle factor in the cylinder with a metric of length $\delta(R_l)$. By the same argument as in Proposition \ref{end_c}, it follows that
$$\mathrm{length}(c_l)\ge \mathrm{length}(F(c_l))\ge 2\left(R_l^2+\dfrac{1}{4}\delta(R_l)^2l^2\right)^{1/2}.$$
Then $\mathrm{length}(c_l)\le \mathrm{length}(\sigma_l)$ yields
$$R_l^2+\dfrac{1}{4}\delta(R_l)^2l^2\le r_l^2+2\delta(r_l)r_l l+\dfrac{1}{4}\delta^2(r_l)l^2.$$
Recall that $l\sim \alpha^{-1}\ln^{\alpha+1}(r^2_l)\cdot r_l$ and $\delta(r_l)\sim \ln^{-\alpha}(r_l^2)$ as $l\to\infty$. Thus two sides of the above inequality have
\begin{align*}
\mathrm{RHS}&\sim r_l^2+2\alpha^{-1}\ln(r_l^2)r_l^2+\dfrac{1}{4\alpha^2}\ln^2(r_l^2)r_l^2\sim\alpha^{-2}\ln^2(r_l)r_l^2,\\
\mathrm{LHS}&\sim R_l^2+\dfrac{1}{4\alpha^2}\ln^{-2\alpha}(R_l^2)\ln^{2\alpha+2}(r_l^2)r_l^2.
\end{align*}

We claim that 
$$\limsup_{l\to\infty} \dfrac{R_l}{\ln(r_l)r_l}=0.$$
Suppose that there is a subsequence $l\to\infty$ such that $R_l\sim C\ln(r_l)r_l$ for some $C>0$. Then for this subsequence,
$$\mathrm{LHS}\sim C^2\ln^2(r_l)r_l^2+\dfrac{1}{4\alpha^2}\cdot\dfrac{\ln^{2\alpha+2}(r_l^2)r_l^2}{\ln^{2\alpha}(C^2\ln^2(r_l)r_l^2)}\sim\left(C^2+\alpha^{-2}\right)\ln^2(r_l)r_l^2.$$
This contradicts the inequality $\mathrm{LHS}\le \mathrm{RHS}$ as $l\to\infty$. We have verified the claim.

Consequently, as $l\to\infty$,
$$\dfrac{d_H(x,c_l)}{\mathrm{length}(c_l)}\le \dfrac{R_l}{\delta(R_l)l}\sim \dfrac{R_l\ln^{\alpha}(R_l)}{2\alpha^{-1}\ln^{\alpha+1}(r_l)r_l}\le \dfrac{\epsilon_l \ln(r_l)r_l\cdot \ln^\alpha(r_l^2)}{2\alpha^{-1}\ln^{\alpha+1}(r_l)r_l}\to 0,$$
where $\epsilon_l\to 0$ as $l\to\infty$. This shows that $E(M,x)=0$.

\Addresses	
	
\end{document}